\documentclass[11pt]{amsart}
\usepackage{hyperref}
\usepackage{amsfonts}
\usepackage{amsmath}
\usepackage{amssymb}
\usepackage{amscd}
\usepackage{graphicx}
\usepackage{enumitem}
\usepackage{latexsym}
\usepackage{color}
\usepackage{comment}
\usepackage{epsfig}
\usepackage{amsopn}
\usepackage{xypic}
\usepackage{mathrsfs}
\usepackage{array}
\usepackage[usenames,dvipsnames]{pstricks}
\usepackage{epsfig}
\usepackage{pst-grad}
\usepackage{pst-plot}
\usepackage[space]{grffile}
\usepackage{etoolbox}
\makeatletter
\patchcmd\Gread@eps{\@inputcheck#1 }{\@inputcheck"#1"\relax}{}{}
\makeatother

\hypersetup{pdfpagemode=UseNone}
\usepackage{booktabs}
\usepackage{longtable}
\theoremstyle{plain}
\newtheorem{lemma}{Lemma}[section]
\newtheorem*{theorem*}{Theorem}
\newtheorem*{lemma*}{Lemma}
\newtheorem*{proposition*}{Proposition}
\newtheorem*{conjecture*}{Conjecture}
\newtheorem*{corollary*}{Corollary}
\newtheorem*{problem*}{Problem}
\newtheorem{theorem}[lemma]{Theorem}

\newtheorem{corollary}[lemma]{Corollary}
\newtheorem{proposition}[lemma]{Proposition}

\theoremstyle{definition}
\newtheorem{definition}[lemma]{Definition}

\newtheorem{remark}[lemma]{Remark}

\newcommand{\CC}{\mathbb{C}}

\newcommand{\RR}{\mathbb{R}}

\newcommand{\ZZ}{\mathbb{Z}}

\newcommand{\PP}{\mathbb{P}}

\DeclareMathOperator{\Gr}{Gr}

\DeclareMathOperator{\sHom}{\mathcal{H}\kern -.5pt\mathit{om}}
\DeclareMathOperator{\sTor}{\mathcal{T}\kern -1.5pt\mathit{or}}

\begin{document}

\date{\today}
\author[Neelarnab Raha]{Neelarnab Raha}
\address{Department of Mathematics, The Pennsylvania State University, University Park, PA 16802}
\email{neelraha@psu.edu}


\title{Ample cones of Hilbert schemes of points on hypersurfaces in $\PP^3$}

\begin{abstract}
Let $X$ be a very general degree $d\geq 5$ hypersurface in $\PP^3$. We compute the ample cone of the Hilbert scheme $X^{[n]}$ of $n$ points on $X$ for various small values of $n$ (the answer is already known for large $n$). We obtain complete answers in some cases and find lower bounds in certain others. We also observe that in the case of $X^{[2]}$ for quintic hypersurfaces $X$, the existence (or absence) of hyperplane sections with points of high multiplicity also plays a role in the answer to the question at hand, in contrast with cases known earlier. Finally, in the case that a degree $d\geq3$ smooth hypersurface $X$ contains a line, we compute the nef cone of $X^{[n]}$ in a slice of the N\'{e}ron-Severi space.
\end{abstract}
\dedicatory{To my Mom, who will be in my heart forever}
\maketitle

\setcounter{tocdepth}{1}
\tableofcontents

\section{Introduction}\label{sec:Intro}

Given any projective variety $Y$, its ample cone inside the N\'{e}ron-Severi space carries information about the embeddings of $Y$ into projective spaces. Its closure is the nef cone of $Y$, denoted $\mbox{Nef}(Y)$. The ample cone is the interior of the nef cone. Hence, calculating the nef cone and calculating the ample cone are equivalent problems. It is well-known (see \cite{Lazarsfeld}) that $\mbox{Nef}(Y)$ is dual to the Kleiman-Mori cone of curves, and it is an interesting question to compute the ample and the nef cones of varieties.

In this paper, we work over the field of complex numbers. A degree $d$ hypersurface $X$ in $\PP^3$ is said to be \emph{$NL$-general} if $X$ is smooth and has Picard rank $1$, with the hyperplane class generating $\mbox{Pic}(X)$. We compute the nef cones of Hilbert schemes $X^{[n]}$ of $n\geq2$ points on $X$ for $NL$-general hypersurfaces $X$ of fixed degree $d\geq5$, for certain small values of $n$. We obtain exact answers in some cases, and find lower bounds in some other cases.

If $X$ is a degree $d\geq3$ smooth hypersurface in $\PP^3$ containing a line (and is therefore not $NL$-general), we calculate the nef cone of $X^{[n]}$ in a two-dimensional slice of the N\'{e}ron-Severi space, for any $n\geq2$.

\subsection{What is known}

Given an $NL$-general degree $d$ hypersurface $X$ in $\PP^3$, the problem of finding the nef cone of $X^{[n]}$ has been studied by several authors (see \cite{Bayer-Macri-K3}, \cite{Bolognese et al.}, \cite{CH-Israel}) by invoking Bridgeland stability conditions. We first briefly describe their findings below.

Let $H^{[n]}$ be the numerical equivalence class of the divisor on $X^{[n]}$ induced by the hyperplane section divisor $H$ on $X$, and let $B^{[n]}$ be the numerical equivalence class of the divisor of non-reduced schemes in $X^{[n]}$. The N\'{e}ron-Severi space $N^1\!\left(X^{[n]}\right)$ is spanned by $H^{[n]}$ and $B^{[n]}$ (see \cite[Lemma 3.1]{Huizenga12}). The fact that $H^{[n]}$ is an extremal nef divisor is fairly easy to see, since it contracts curves on $X^{[n]}$.

In \cite{Bayer-Macri-K3}, the authors study moduli spaces of stable sheaves on $K3$ surfaces using Bridgeland stability conditions. As a corollary (\cite[Lemma 13.3\footnote{\href{https://arxiv.org/abs/1301.6968}{arXiv:1301.6968v4 [math.AG]} has a more updated version of this result.}]{Bayer-Macri-K3}) of their results, it follows that for an $NL$-general degree $4$ surface $X$ in $\PP^3$, the nef cone of the Hilbert scheme $X^{[2]}$ is spanned by $H^{[2]}$ together with $$\frac{3}{4}H^{[2]}-\frac{1}{2}B^{[2]}\,.$$

In \cite[Proposition 4.5]{Bolognese et al.}, the authors show that the nef cone of $X^{[n]}$ is spanned by $H^{[n]}$ and \begin{equation}\label{eq:BHL+Formula}
	\left(\frac{d}{2}-\frac{3}{2}+\frac{n}{d}\right)H^{[n]}-\frac{1}{2}B^{[n]}\,,
\end{equation} whenever $n\geq d-1$, where $X$ is an $NL$-general degree $d\geq4$ hypersurface in $\PP^3$. The hard part is to show nefness of the latter divisor, and the crucial ingredient for this is the Positivity Lemma of Bayer and Macr\`{i} (\cite[Lemma 3.3]{Bayer-Macri-PositivityLemma})\footnote{This is also used in \cite{Bayer-Macri-K3} for the case of $d=4,\,n=2$.}. Its extremality is shown by exhibiting a curve on $X^{[n]}$ orthogonal to the divisor class.

We should mention that the above results completely answer the question of the computation of nef cones of the Hilbert schemes $X^{[n]}$ of $n\geq2$ points on $NL$-general quartic hypersurfaces $X$ in $\PP^3$.

\subsection{What we do in this paper}

In this paper, we tackle cases in which the Bridgeland stability method mentioned above does not seem to apply. We resort to more classical techniques, and find lower bounds on the edge of the nef cone of $X^{[n]}$ other than the edge spanned by $H^{[n]}$ by constructing curves $\Gamma$ on $X^{[n]}$ with high values of the ratio $\left(\Gamma\!\cdot\!B^{[n]}\right)\!/\!\left(\Gamma\!\cdot\!H^{[n]}\right)$. In certain cases, we argue the nefness of the numerical class orthogonal to the constructed curve, primarily by pulling back  nef divisors via maps. In such cases, we then have complete answers to the question of computation of the nef cone of the Hilbert scheme.

We summarize our main results below.

\begin{theorem}
	Let $X$ be an $NL$-general degree $d$ hypersurface in $\PP^3$ for some integer $d\geq5$. Then the nef cone of $X^{[d-2]}$ is spanned by $H^{[d-2]}$ and $$\alpha H^{[d-2]}-\frac{1}{2}B^{[d-2]}$$ for some real number $$\alpha\geq\frac{d^2-d-6}{2d}\,.$$
\end{theorem}

To prove the above result, we exhibit a curve orthogonal to $$\left(\frac{d^2-d-6}{2d}\right)\!H^{[d-2]}-\frac{1}{2}B^{[d-2]}\,.$$ We then show that for $d=5,6$, the class $$\left(\frac{d^2-d-6}{2d}\right)\!H^{[d-2]}-\frac{1}{2}B^{[d-2]}$$ is in fact nef on $X^{[d-2]}$. Thus, we obtain the following result.

\begin{theorem}
	Let $X$ be an $NL$-general degree $d$ hypersurface in $\PP^3$ for $d=5$ or $6$. \begin{enumerate}
		\item If $d=5$, then the nef cone of $X^{[3]}$ is spanned by $H^{[3]}$ and $$\frac{7}{5}H^{[3]}-\frac{1}{2}B^{[3]}\,.$$
		\item If $d=6$, then the nef cone of $X^{[4]}$ is spanned by $H^{[4]}$ and $$2H^{[4]}-\frac{1}{2}B^{[4]}\,.$$
	\end{enumerate}
\end{theorem}

Before we mention the remaining results, it is imperative to make the following definitions.

\begin{definition}\label{def:SurfaceFlex}
	Let $X$ be a degree $d\geq3$ hypersurface in $\PP^3$, and let $p\in X$.
	\begin{enumerate}
		\item We say that $p$ is a \emph{surface $r$-flex} of $X$ (for some $3\leq r\leq d$) if there is a hyperplane $\Lambda$ in $\PP^3$ containing $p$ such that the hyperplane section of $X$ given by $\Lambda$ has $p$ as a singular point of multiplicity at least $r$.
		\item We say that $p$ is an \emph{ordinary surface $r$-flex} of $X$ if there is a hyperplane $\Lambda$ in $\PP^3$ containing $p$ such that the hyperplane section of $X$ given by $\Lambda$ has $p$ as an ordinary singularity of multiplicity exactly $r$.
	\end{enumerate}
\end{definition}

Next, we argue that given $r\geq3$, once $d$ becomes large enough, there are $NL$-general surfaces of degree $d$ that have ordinary surface $r$-flexes. This ensures that the hypothesis of the following theorem is not an empty condition.

\begin{theorem}
	Let $X$ be an $NL$-general degree $d$ hypersurface in $\PP^3$ for some $d\geq5$, and let $r\geq3$. Suppose $X$ has an ordinary surface $r$-flex. Then the nef cone of $X^{[d-r]}$ is spanned by $H^{[d-r]}$ and $$\alpha H^{[d-r]}-\frac{1}{2}B^{[d-r]}$$ for some real number $$\alpha\geq\frac{d^2-d-r^2-r}{2d}\,.$$
\end{theorem}

The ordinary surface $r$-flex gives us a hyperplane section of $X$ with a point of multiplicity $r$, and this allows us to get hold of a curve on $X^{[d-r]}$ by ``spinning'' a line through the multiplicity $r$ point of the hyperplane section, and taking the residual length $d-r$ scheme of intersection of the line with $X$. This curve is orthogonal to $$\left(\frac{d^2-d-r^2-r}{2d}\right)\!H^{[d-r]}-\frac{1}{2}B^{[d-r]}\,.$$

For the case of quintics, upon showing the nefness of $$\frac{4}{5}H^{[2]}-\frac{1}{2}B^{[2]}$$ by pulling back a nef divisor via some morphisms, we obtain the following result as a corollary of the previous one.

\begin{corollary}
	Let $X$ be an $NL$-general quintic hypersurface in $\PP^3$ that has an ordinary surface $3$-flex. Then the nef cone of $X^{[2]}$ is spanned by $H^{[2]}$ and $$\frac{4}{5}H^{[2]}-\frac{1}{2}B^{[2]}\,.$$
\end{corollary}

We also analyze the case when an $NL$-general quintic does not have a surface $3$-flex, and obtain the following result.

\begin{theorem}
	Let $X$ be an $NL$-general quintic hypersurface in $\PP^3$ having no surface $3$-flexes. \begin{enumerate}
		\item The class $$\frac{4}{5}H^{[2]}-\frac{1}{2}B^{[2]}$$ is ample on $X^{[2]}$.\\
		
		\item The nef cone of $X^{[2]}$ is spanned by $H^{[2]}$ and $$\alpha H^{[2]}-\frac{1}{2}B^{[2]}$$ for some real number $\alpha$ with $\dfrac{7}{10}\leq\alpha<\dfrac{4}{5}$.
	\end{enumerate}
\end{theorem}

An interesting phenomenon should be pointed out here. In each of the cases covered by \cite{Bolognese et al.}, all $NL$-general hypersurfaces $X$ of a fixed degree had the same answer for the nef cone of $X^{[n]}$. However, this is not the case for $X^{[2]}$ for quintic hypersurfaces $X$, as we do not have a uniform answer for all $NL$-general quintics.

Finally, we compute the nef cones of Hilbert schemes of points on a smooth hypersurface containing a line, which is on the other extreme compared to $NL$-general ones.

\begin{theorem}
	Let $X$ be a degree $d\geq3$ smooth hypersurface in $\PP^3$ containing a line, and let $n\geq2$. Then the classes $H^{[n]}$ and $$(n-1)H^{[n]}-\frac{1}{2}B^{[n]}$$ span the nef cone of $X^{[n]}$ in the slice of the N\'{e}ron-Severi space spanned by $H^{[n]}$ and $B^{[n]}$.
\end{theorem}

\subsection{The outline of this paper}

In \S\ref{sec:Nef-X-(d-2)}, we exhibit a lower bound on the nef cone of $X^{[d-2]}$ for an $NL$-general degree $d$ hypersurface $X$. After that, we show in \S\ref{sec:Nef-X3-X4} that this lower bound is actually the answer in the cases of $X^{[3]}$ for quintics and $X^{[4]}$ for sextics. In \S\ref{sec:Nef-X-(d-r)}, we compute a lower bound on the nef cone of $X^{[d-r]}$ for any $NL$-general degree $d$ hypersurface $X$ that has an ordinary surface $r$-flex. In \S\ref{sec:Nef-X2}, we discuss the implications of our results from \S\ref{sec:Nef-X-(d-r)} to the case of $X^{[2]}$ for quintics, thus calculating the nef cone assuming $X$ has an ordinary surface $3$-flex. We also analyze the case when the quintic $X$ does not contain any surface $3$-flexes. Finally in \S\ref{sec:Nef-NLspecial}, we compute the nef cones of the Hilbert schemes of points on a smooth hypersurface containing a line.

\subsection*{Acknowledgment} The author is highly indebted to Prof.\! Jack Huizenga for his guidance throughout this project. The author would also like to thank Prof.\! John Ottem and Prof.\! Izzet Coskun for their valuable feedback which led to several improvements in this paper.

\section{Preliminaries}\label{sec:Prelims}

Let $U_d$ denote the locus (in the projective space $\PP^{\binom{d+3}{3}-1}$ of all degree $d$ surfaces in $\PP^3$) of smooth degree $d$ surfaces. Recall that for $d\geq4$, there is a \emph{Noether-Lefschetz locus} $NL_d\subset U_d$\,, which is a countable union of proper closed subvarieties of $U_d$\,, such that the Picard group of any $X\in U_d\backslash NL_d$ is isomorphic to $\ZZ$, generated by the hyperplane class (see \cite[Theorem 7.5.4]{Carlson et al.}). We now make the following definition.

\begin{definition}
	We say that a surface $X$ of degree $d\geq4$ in $\PP^3$ is \emph{$NL$-general} if $X$ lies in $U_d\backslash NL_d$\,.
\end{definition}

\subsection{Preliminaries about $N^1\!\left(X^{[n]}\right)$}

Let $X$ be a smooth hypersurface of degree $d\geq4$ in $\PP^3$. For any integer $n\geq2$, let $X^{[n]}$ be the Hilbert scheme of $n$ points on $X$, parametrizing zero-dimensional closed subschemes of $X$ of length $n$. From \cite{Fogarty68}, we know that $X^{[n]}$ is a smooth projective variety of dimension $2n$, and the Hilbert-Chow morphism $X^{[n]}\rightarrow X^{(n)}$ is a resolution of singularities, where $X^{(n)}$ is the $n^{\mbox{\tiny th}}$ symmetric power of $X$.

Given any line bundle $L$ on $X$, we get an $S_n$-equivariant line bundle $L^{\boxtimes n}$ on $X^n$, which descends to a line bundle $L^{(n)}$ on the symmetric product $X^{(n)}$. The pullback of $L^{(n)}$ under the Hilbert-Chow morphism $X^{[n]}\rightarrow X^{(n)}$ is a line bundle on $X^{[n]}$, denoted by $L^{[n]}$. This construction preserves linear equivalence. Geometrically, if $L=\mathcal{O}_X(D)$ for some reduced effective divisor $D$ on $X$, then $L^{[n]}$ is represented by the locus $D^{[n]}$ of elements $Z\in X^{[n]}$ such that $Z$ meets $D$.

Let $B^{[n]}$ be the class in $\mbox{Pic}(X^{[n]})$ of the exceptional divisor of the Hilbert-Chow morphism. Equivalently, $B^{[n]}$ is the class of the locus of non-reduced schemes $Z\in X^{[n]}$.

Since $q(X)=0$, Fogarty shows (see \cite{Fogarty73}) that $\mbox{Pic}(X^{[n]})\cong\mbox{Pic}(X)\oplus\ZZ\!\left(B^{[n]}/2\right)$, where $\mbox{Pic}(X)$ is viewed as a subgroup of $\mbox{Pic}(X^{[n]})$ via the map $L\mapsto L^{[n]}$. Tensoring by $\RR$, we see that the N\'{e}ron-Severi space $N^1\!\left(X^{[n]}\right)$ is spanned by $N^1(X)$ and the class of $B^{[n]}$.

If $H$ is the hyperplane class on $X$, the class $H^{[n]}$ on $X^{[n]}$ is nef since it is the pullback of the ample line bundle $H^{(n)}$ via the Hilbert-Chow morphism, and it has curves orthogonal to it. So, if $X$ is $NL$-general, then $H^{[n]}$ spans one edge of the nef cone of $X^{[n]}$ and we only need to determine the other edge of the nef cone.

\subsection{The residuation isomorphism}\label{subsec:Residuation}

Given any $NL$-general degree $d\geq4$ hypersurface $X$, let $S_X\subset X^{[d-2]}$ be the locus of collinear subschemes, and let $i:S_X\hookrightarrow X^{[d-2]}$ be the inclusion morphism. We have a \textit{residuation isomorphism} $\psi:X^{[2]}\xrightarrow{\sim}S_X$ sending any $Z\in X^{[2]}$ to its residual scheme with respect to $\ell\cap X$, where $\ell$ is the line in $\PP^3$ spanned by $Z$. Note that this implies $S_X$ is a smooth, $4$-dimensional closed subvariety of $X^{[d-2]}$\,.

The residuation isomorphism will be very useful for us. For example, it will allow us to pullback known nef divisors on one of $X^{[2]}$ and $S_X$ to the other, thus giving us nef divisors on the other variety.

\subsection{Some curves}\label{subsec:SomeCurves}

Let $n\geq2$ be any integer. Fix a hypersurface $X$ in $\PP^3$, and let $d$ be its degree.

We obtain a curve class $\Gamma_{[n]}$ on $X^{[n]}$ by fixing $n-1$ distinct general points $p_{_1},\ldots,p_{_{n-1}}$ of $X$, and ``spinning'' a tangent vector to $X$ at $p_{_1}$\,. Then $\Gamma_{[n]}\!\cdot\!H^{[n]}=0$ and $\Gamma_{[n]}\!\cdot\!B^{[n]}=-2$ (see \cite[Chapter 3]{Huizenga12}).

By using the residuation isomorphism, $\Gamma_{[2]}$ gives us a curve $\psi\!\left(\Gamma_{[2]}\right)$ on $S_X$. It can be described as follows\,: take the hyperplane section $C:=X\cap\mathbb{T}_{p_{_1}}\!X$ of $X$ cut out by the tangent plane to $X$ at $p_{_1}$, ``spin'' a line through the node $p_{_1}$ of the plane curve $C$, and take the residual length $d-2$ scheme of intersection of the spinning line with $C$.

We shall also need two more types of curves on $X^{[n]}$\,:

\begin{itemize}

	\item $\Phi_{[n]}$ obtained by fixing $n-1$ points of $X$ and letting an $n^{\mbox{\tiny th}}$ point move along a general hyperplane section of $X$. Note that $\Phi_{[n]}\!\cdot\!H^{[n]}=d$ and $\Phi_{[n]}\!\cdot\!B^{[n]}=0$.\\[-5pt]

	\item $\Psi_{[n]}$ obtained by fixing $n-2$ points of $X$ and a general point of a general hyperplane section of $X$, and letting an $n^{\mbox{\tiny th}}$ point move along that hyperplane section. Using arguments similar to those given in \cite[Chapter 3]{Huizenga12}, it can be shown that $\Psi_{[n]}\!\cdot\!H^{[n]}=d$ and $\Psi_{[n]}\!\cdot\!B^{[n]}=2$.

\end{itemize}

\section{The nef cone of $X^{[d-2]}$ for a degree $d$ hypersurface}\label{sec:Nef-X-(d-2)}

Let $X$ be an $NL$-general degree $d$ hypersurface in $\PP^3$, for some $d\geq5$.

Recall from Section \ref{subsec:Residuation} above that $S_X\subset X^{[d-2]}$ is defined to be the locus of collinear subschemes of $X$, and we have a \emph{residuation isomorphism} $\psi:X^{[2]}\xrightarrow{\sim}S_X$. Let $i:S_X\hookrightarrow X^{[d-2]}$ be the inclusion morphism.

By exhibiting a dual curve, we now find a lower bound on the nef cone of $X^{[d-2]}$. Let $\Gamma_{[2]}$ be the curve on $X^{[2]}$ from Section \ref{subsec:SomeCurves}, contracted by the Hilbert-Chow morphism. Then we have the following intersection numbers.

\begin{lemma}\label{lemma:i_of_psi_of_Gamma2-IntersectionNumbers}
	We have $$i\!\left(\psi\!\left(\Gamma_{[2]}\right)\right)\!\cdot\!B^{[d-2]}=d^2-d-6$$ and $$i\!\left(\psi\!\left(\Gamma_{[2]}\right)\right)\!\cdot\!H^{[d-2]}=d.$$
\end{lemma}

\begin{proof}
	The geometric description of the curve $\psi\!\left(\Gamma_{[2]}\right)$, as explained in Section \ref{subsec:SomeCurves} above, gives us a surjective mapping $\tilde{C}\rightarrow\PP^1$ of degree $(d-2)$, where $\tilde{C}$ is the normalization of the curve $$C:=X\cap\mathbb{T}_{p_{_1}}\!X.$$
	
	The projective plane curve $C$ has geometric genus $$\binom{d-1}{2}-1$$ since it is of degree $d$ and has a node at $p_{_1}$ as its only singularity. By Riemann-Hurwitz, the degree of the ramification divisor of the above map $\tilde{C}\rightarrow\PP^1$ is $d^2-d-6$. It follows that $$i\!\left(\psi\!\left(\Gamma_{[2]}\right)\right)\!\cdot\!B^{[d-2]}=d^2-d-6.$$
	
	Since $H^{[d-2]}$ is represented by the locus of schemes $Z\in X^{[d-2]}$ such that $Z$ meets a fixed general hyperplane section of $X$, we see that $$i\!\left(\psi\!\left(\Gamma_{[2]}\right)\right)\!\cdot\!H^{[d-2]}=d.$$
\end{proof}

From the above lemma, it follows that $$i\!\left(\psi\!\left(\Gamma_{[2]}\right)\right)\!\cdot\!\left(\left(\frac{d^2-d-6}{2d}\right)\!H^{[d-2]}-\frac{1}{2}B^{[d-2]}\right)=0.$$ This gives us a lower bound on the nef cone.

\begin{theorem}\label{theorem:Nef-X(d-2)}
	Let $X$ be a $NL$-general degree $d$ hypersurface in $\PP^3$, for some integer $d\geq5$. Then the nef cone of the Hilbert scheme $X^{[d-2]}$ is spanned by $H^{[d-2]}$ and $$\alpha H^{[d-2]}-\frac{1}{2}B^{[d-2]}$$ for some real number $$\alpha\geq\frac{d^2-d-6}{2d}\,.$$
\end{theorem}

\section{The nef cone of $X^{[3]}$ for a quintic and $X^{[4]}$ for a sextic  hypersurface}\label{sec:Nef-X3-X4}

In this section, we shall argue that for $d=5,6$, the class $$\left(\frac{d^2-d-6}{2d}\right)\!H^{[d-2]}-\frac{1}{2}B^{[d-2]}$$ appearing in Theorem \ref{theorem:Nef-X(d-2)} is in fact nef on $X^{[d-2]}$, whence it must be an extremal nef class. In both cases, it will be fairly easy to see that it is nef outside of $S_X$, and on $S_X$ we use residuation to show nefness.

In the process, we shall be requiring two rational maps of a similar flavor, both defined outside the locus $S_X$\,. Namely, for a sextic $X$, we shall use the rational map $$X^{[4]}\dashrightarrow \Gr\!\left(6,H^0\!\left(\mathcal{O}_{\PP^3}(2)\right)\right)$$ taking any non-collinear $Z\in X^{[4]}$ to the subspace of quadratic forms vanishing on $Z$. And for a quintic $X$, we shall use the rational map $$X^{[3]}\dashrightarrow \PP^{3*}=\Gr\!\left(1,H^0\!\left(\mathcal{O}_{\PP^3}(1)\right)\right)$$ taking any non-collinear $Z\in X^{[3]}$ to the hyperplane in $\PP^3$ that it spans. For a more uniform description, this map can also be described as taking $Z$ to the locus of linear forms vanishing on $Z$.

Before specializing to the cases of quintics or sextics, let us prove a general lemma calculating the classes of some pullbacks under the residuation isomorphism $\psi:X^{[2]}\xrightarrow{\sim}S_X$\,, and obtain a corollary of these calculations.

\begin{lemma}\label{lemma:Pullbacks_01}
	With notation as above, we have $$\psi^*i^*B^{[d-2]}=(d^2-3d)H^{[2]}-\frac{(d^2-d-6)}{2}B^{[2]}$$ and $$\psi^*i^*H^{[d-2]}=(d-1)H^{[2]}-\frac{d}{2}B^{[2]}\,.$$
\end{lemma}

\begin{proof}
	Let us write $$\psi^*i^*B^{[d-2]}=\alpha H^{[2]}-\frac{\beta}{2}B^{[2]}\mbox{ and }\psi^*i^*H^{[d-2]}=\alpha'H^{[2]}-\frac{\beta'}{2}B^{[2]}$$ for some real numbers $\alpha,\beta,\alpha',\beta'$. First, let us compute $\alpha$ and $\beta$.
	
	Recall the curve $\Gamma_{[2]}$ from Section \ref{subsec:SomeCurves}. From Lemma \ref{lemma:i_of_psi_of_Gamma2-IntersectionNumbers}, it follows that $$\Gamma_{[2]}\!\cdot\!\left(\psi^*i^*B^{[d-2]}\right)=i\!\left(\psi\!\left(\Gamma_{[2]}\right)\right)\!\cdot\!B^{[d-2]}=d^2-d-6.$$ Since $\Gamma_{[2]}\!\cdot\!H^{[2]}=0$ and $\Gamma_{[2]}\!\cdot\!B^{[2]}=-2$, we get $\beta=d^2-d-6$.
	
	Next, we compute $\alpha$. Let $\Psi_{[2]}$ be the curve on $X^{[2]}$ obtained by fixing a general point $p$ of a general hyperplane section $\mathcal{C}$ of $X$, and letting a second point move along $\mathcal{C}$. Projection of $\mathcal{C}$ from $p$ gives a map $\eta:\mathcal{C}\rightarrow\PP^1$ of degree $d-1$. Since $p$ is a general point of $\mathcal{C}$, no line passing through $p$ is going to be a double tangent to $\mathcal{C}$. It follows that for every branch point $t$ of $\eta$, there are $d-3$ points $q_{_{t,1}}\,,\ldots,q_{_{t,d-3}}$ in the fiber of $\eta$ over $t$ for which the residuation of the scheme $\{p,q_{_{t,i}}\}$ is non-reduced ($i=1,\ldots,d-3$). Also note that for each such $t$, the scheme $\{p,q_{_{t,i}}\}$ belongs to $\Psi_{[2]}$\,, for $i=1,\ldots,d-3$. Since $\mathcal{C}$ has genus $\binom{d-1}{2}$, by Riemann-Hurwitz the degree of the ramification divisor of $\eta$ is $d^2-d-2$.  Since $i\!\left(\psi\!\left(\Psi_{[2]}\right)\right)\!\cdot\!B^{[d-2]}$ counts the number of points of $\Psi_{[2]}$ whose residuation is non-reduced, it follows that $$\Psi_{[2]}\!\cdot\!\left(\psi^*i^*B^{[d-2]}\right)=i\!\left(\psi\!\left(\Psi_{[2]}\right)\right)\!\cdot\!B^{[d-2]}=(d-3)\cdot(d^2-d-2).$$ Recalling that $\Psi_{[2]}\!\cdot\!H^{[2]}=d$ and $\Psi_{[2]}\!\cdot\!B^{[2]}=2$, we get $\alpha=d^2-3d$.
	
	Hence, $$\psi^*i^*B^{[d-2]}=(d^2-3d)H^{[2]}-\frac{(d^2-d-6)}{2}B^{[2]}\,.$$
	
	Next, we compute $\alpha'$ and $\beta'$. From Lemma \ref{lemma:i_of_psi_of_Gamma2-IntersectionNumbers}, we have $$\Gamma_{[2]}\!\cdot\!\left(\psi^*i^*H^{[d-2]}\right)=i\!\left(\psi\!\left(\Gamma_{[2]}\right)\right)\!\cdot\!H^{[d-2]}=d\,.$$ Thus, $\beta'=d$.
	
	Finally, observe that a general hyperplane $\mathcal{H}$ in $\PP^3$ meets $\mathcal{C}$ in $d$ distinct points (say $x_{_1}\,,\ldots,x_{_d}$), and does not pass through $p$. By all the generality hypotheses, it follows that for each $1\leq i\leq d$, the residuation of the scheme $\{p,x_{_i}\}$ consists of $d-2$ distinct points, say $y_{_{i,1}}\,,\ldots,y_{_{i,d-2}}$\,, and none of these is equal to any of the $x_{_j}$'s. Then we have exactly $d(d-2)$ points $\{p,y_{_{i,j}}\}\in\Psi_{[2]}$ ($1\leq i\leq d$ and $1\leq j\leq d-2$) for which the residual scheme meets $\mathcal{H}$. Hence,  $$\Psi_{[2]}\!\cdot\!\left(\psi^*i^*H^{[d-2]}\right)=i\!\left(\psi\!\left(\Psi_{[2]}\right)\right)\!\cdot\!H^{[d-2]}=d(d-2).$$
	
	Therefore, $\alpha'=d-1$, whence $$\psi^*i^*H^{[d-2]}=(d-1)H^{[2]}-\frac{d}{2}B^{[2]}\,.$$
	
\end{proof}

We then get the following result.

\begin{lemma}\label{lemma:NefnessOnS_X}
	Let the hypotheses be as in Lemma \ref{lemma:Pullbacks_01}. Then the class $$i^*\!\left(\frac{(d^2-d-6)}{2d}H^{[d-2]}-\frac{1}{2}B^{[d-2]}\right)\in N^1\!\left(S_X\right)$$ is nef on $S_X$\,.
\end{lemma}

\begin{proof}
	From Lemma \ref{lemma:Pullbacks_01} it follows that $$\psi^*i^*\!\left(\frac{(d^2-d-6)}{2d}H^{[d-2]}-\frac{1}{2}B^{[d-2]}\right)=\frac{(d^2-5d+6)}{2d}H^{[2]}\,.$$ Thus, $$i^*\!\left(\frac{(d^2-d-6)}{2d}H^{[d-2]}-\frac{1}{2}B^{[d-2]}\right)=\left(\psi^{-1}\right)^*\!\left(\frac{(d^2-5d+6)}{2d}H^{[2]}\right)$$ is a nef class on $S_X$.
\end{proof}

\begin{remark}
	From Lemma \ref{lemma:Pullbacks_01} and (the proof of) Lemma \ref{lemma:NefnessOnS_X}, we see that the residuation isomorphism $\psi:X^{[2]}\xrightarrow{\sim}S_X$ ``interchanges'' the ray spanned by the class $H^{[n]}$ ($n=2,d-2$) of the divisor inducing the Hilbert-Chow morphism with a non-trivial ray spanned by some other class.
\end{remark}

\subsection{The nef cone of $X^{[3]}$ for a quintic}\label{subsec:Nef-X3}

In this subsection, $X$ will be an $NL$-general quintic hypersurface in $\PP^3$, and let $\varphi_{_1}:X^{[3]}\dashrightarrow\PP^{3*}$ be the rational map taking any non-collinear $Z$ to the hyperplane in $\PP^3$ that it spans. So $S_X$ is precisely the indeterminacy locus of $\varphi_{_1}$\,.

Since $S_X\cong X^{[2]}$ has codimension $2$ in $X^{[3]}$, restriction gives an isomorphism $$\mbox{Pic}(X^{[3]})\xrightarrow{\sim}\mbox{Pic}(X^{[3]}\backslash S_X).$$ Therefore, it makes sense to talk of the pullback of a divisor (and numerical) class from $\PP^{3*}$ to $X^{[3]}$ via $\varphi_{_1}$\,.

Our current goal is to prove the following theorem. \begin{theorem}\label{theorem:Nef-X3} Let $X$ be an $NL$-general quintic hypersurface in $PP^3$. Then the nef cone of the Hilbert scheme $X^{[3]}$ is spanned by $H^{[3]}$ and $$\frac{7}{5}H^{[3]}-\frac{1}{2}B^{[3]}\,.$$
\end{theorem}

Briefly, the idea is as follows. We first show in Lemma \ref{lemma:NefnessOutsideS_X} that for any real number $a\geq1$, the numerical equivalence class $$aH^{[3]}-\frac{1}{2}B^{[3]}$$ non-negatively intersects any irreducible curve on $X^{[3]}$ that is not contained in $S_X$\,. In particular, this applies to $$\frac{7}{5}H^{[3]}-\frac{1}{2}B^{[3]}\,.$$ By Lemma \ref{lemma:NefnessOnS_X}, we already know that its restriction to $S_X$ is nef. Theorem \ref{theorem:Nef-X(d-2)} then completes the argument. We now present the details. We shall need a couple of lemmas.

\begin{lemma}\label{lemma:PullbackOfLfromP3*}
	Let $L$ be the hyperplane class in $\PP^{3*}$. Then $$\varphi_{_1}^*L=H^{[3]}-\frac{1}{2}B^{[3]}\,.$$
\end{lemma}

\begin{proof}
	
	Note that $L$ is represented by the class of all hyperplanes in $\PP^3$ containing a fixed general point $w\in\PP^3$. Write $$\varphi_{_1}^*L=\alpha H^{[3]}-\frac{\beta}{2}B^{[3]}$$ for some real numbers $\alpha$ and $\beta$.
	
	As in Section \ref{subsec:SomeCurves}, let $\Gamma_{[3]}$ be the curve on $X^{[3]}$ obtained by fixing $2$ distinct general points $p_{_1}$ and $p_{_2}$ on $X$ and ``spinning'' a tangent vector to $X$ at $p_{_1}$. Note that the cardinality of the set $$\{Z\in\Gamma_{[3]}\,:\,\mbox{the plane spanned by }Z\mbox{ contains }w\}$$ is $1$, with its sole member being the scheme $Z$ consisting of $p_{_1}$, $p_{_2}$ and the tangent vector to $X$ at $p_{_1}$ given by the intersection of the tangent plane to $X$ at $p_{_1}$ and the hyperplane in $\PP^3$ spanned by $p_{_1},p_{_2}$ and $w$. It follows that $\Gamma_{[3]}\cdot\varphi_{_1}^*L=1$, whence $\beta=1$, since $\Gamma_{[3]}\cdot H^{[3]}=0$ and $\Gamma_{[3]}\cdot B^{[3]}=-2$.
	
	Next, let $\Phi_{[3]}$ be the curve on $X^{[3]}$ obtained by fixing two points $q_{_1}$ and $q_{_2}$ of $X$, and letting a third point $p$ move along a general hyperplane section $\mathcal{C}$ of $X$. Then \begin{flalign*}
		\Phi_{[3]}\cdot\varphi_{_1}^*L&=\#\{Z\in\Phi_{[3]}\,:\,\mbox{the plane spanned by }Z\mbox{ contains }w\}\\[2pt] &=\#\{Z\in\Phi_{[3]}\,:\,\mbox{the plane spanned by }q_{_1}\,,q_{_2}\mbox{ and }w\mbox{ contains the moving point }p\}\\[2pt] &=\mbox{degree of the projective plane curve }\mathcal{C}\\[2pt] &=5.
	\end{flalign*} Since $\Phi_{[3]}\cdot H^{[3]}=5$ and $\Phi_{[3]}\cdot B^{[3]}=0$, we see that $\alpha=1$.
	
\end{proof}

This immediately leads to the following observation.

\begin{lemma}\label{lemma:NefnessOutsideS_X}
	For any real number $a\geq1$, the class $aH^{[3]}-\frac{1}{2}B^{[3]}\in N^1\!\left(X^{[3]}\right)$ non-negatively intersects any irreducible curve on $X^{[3]}$ that is not contained in $S_X$\,.
\end{lemma}

\begin{proof}
	Since $L$ is ample on $\PP^{3*}$, it follows from Lemma \ref{lemma:PullbackOfLfromP3*} that $H^{[3]}-\frac{1}{2}B^{[3]}$ is non-negative on any irreducible curve in $X^{[3]}$ that is not contained in $S_X$\,. Since $H^{[3]}$ is nef, we see that $aH^{[3]}-\frac{1}{2}B^{[3]}$ is also non-negative on any irreducible curve in $X^{[3]}$ that is not contained in $S_X$, for any real $a\geq1$.
\end{proof}

We now have all the pieces we need to complete the proof of Theorem \ref{theorem:Nef-X3}.

\begin{proof}[Proof of Theorem \ref{theorem:Nef-X3}]
	From Lemma \ref{lemma:NefnessOnS_X} above, it follows that $$i^*\!\left(\frac{7}{5}H^{[3]}-\frac{1}{2}B^{[3]}\right)=\left(\psi^{-1}\right)^*\!\left(\frac{3}{5}H^{[2]}\right)$$ is a nef class on $S_X$\,.
	
	Noting that for any curve $C$ contained in $S_X$, and any divisor $D$ on $X^{[3]}$, we have $$C\cdot D=i_*(C)\cdot D=C\cdot i^*D,$$ it follows that $$\frac{7}{5}H^{[3]}-\frac{1}{2}B^{[3]}$$ non-negatively intersects any curve contained in $S_X$\,.
	
	Finally, by Lemma \ref{lemma:NefnessOutsideS_X}, we know that $$\frac{7}{5}H^{[3]}-\frac{1}{2}B^{[3]}$$ also non-negatively intersects any other curve on $X^{[3]}$.
	
	Hence, $$\frac{7}{5}H^{[3]}-\frac{1}{2}B^{[3]}$$ is nef on all of $X^{[3]}$. By Theorem \ref{theorem:Nef-X(d-2)}, we know that the nef cone of $X^{[3]}$ cannot extend beyond the ray spanned by this class, whence it must span an extremal ray of the nef cone.
\end{proof}

\begin{remark}\label{remark:Nef-X3}
	We remark here that Theorem \ref{theorem:Nef-X3} shows that in the case of an $NL$-general quintic hypersurface $X$ in $\PP^3$, the nef cone of $X^{[3]}$ indeed extends beyond the ray spanned by $$\frac{8}{5}H^{[3]}-\frac{1}{2}B^{[3]}\,,$$ which is the ray one would get simply by substituting $d=5$ and $n=3$ in the formula (\ref{eq:BHL+Formula}) from Proposition 4.5 of \cite{Bolognese et al.}.
\end{remark}

\subsection{The nef cone of $X^{[4]}$ for a sextic}\label{subsec:Nef-X4}

Let $X$ be an $NL$-general sextic hypersurface in $\PP^3$, and let $$\varphi_{_2}:X^{[4]}\dashrightarrow \mathbb{G}:=\mathbb{G}(5,9)$$ be the rational map taking any non-collinear $Z\in X^{[4]}$ to the locus of quadratic forms on $\PP^3$ vanishing on $Z$.

As before, we have an isomorphism $\mbox{Pic}(X^{[4]})\xrightarrow{\sim}\mbox{Pic}(X^{[4]}\backslash S_X)$ given by restriction. So it makes sense to pullback a divisor (and numerical) class from $\mathbb{G}$ to $X^{[4]}$ via $\varphi_{_2}$.

In this subsection, we prove the following assertion. \begin{theorem}\label{theorem:Nef-X4} Let $X$ be an $NL$-general sextic hypersurface in $\PP^3$. Then the nef cone of the Hilbert scheme $X^{[4]}$ is spanned by $H^{[4]}$ and $$2H^{[4]}-\frac{1}{2}B^{[4]}\,.$$
\end{theorem}

The plan for proving this is essentially the same as for the quintic case discussed in Section \ref{subsec:Nef-X3} above.

\begin{lemma}\label{lemma:PullbackOfHfromGr}
	Let $\sigma_1$ be the Schubert class of all planes in $\mathbb{G}$ that meet a fixed general $3$-dimensional plane $\Lambda\subseteq\PP^9$. Then $$\varphi_{_2}^*\sigma_1=2H^{[4]}-\frac{1}{2}B^{[4]}\,.$$ Consequently, the class $$2H^{[4]}-\frac{1}{2}B^{[4]}\in N^1\!\left(X^{[4]}\right)$$ non-negatively intersects any irreducible curve on $X^{[4]}$ that is not contained in $S_X$\,.
\end{lemma}

\begin{proof}
	
	We start by writing $$\varphi_{_2}^*\sigma_1=\alpha H^{[4]}-\frac{\beta}{2}B^{[4]}$$ for some real numbers $\alpha$ and $\beta$.
	
	Let $\Phi_{[4]}$ be the curve on $X^{[4]}$ obtained by fixing three points $q_{_1},q_{_2},q_{_3}$ of $X$, and letting a fourth point move along a general hyperplane section $\mathcal{C}:=\mathcal{H}\cap X$ of $X$ for some hyperplane $\mathcal{H}\subseteq\PP^3$.
	
	Counting the number of independent conditions imposed on quadrics, note that there is only one quadric hypersurface $Q\subseteq\PP^3$ that contains each of $q_{_1}\,,q_{_2}$ and $q_{_3}$ and also belongs to $\Lambda$\,. Then, since $\mathcal{C}$ is a degree $6$ curve in $\mathcal{H}$, and $\mathcal{H}\cap Q$ is a degree $2$ curve in $\mathcal{H}$, it follows that $\Phi_{[4]}\cdot\varphi_{_2}^*\sigma_1=12$. Hence, $\alpha=2$.
	
	Finally, let $\Psi_{[4]}$ be the curve on $X^{[4]}$ obtained by fixing two points of $X$, and also fixing a general point of a general hyperplane section $\mathcal{C}'$ of $X$, and letting a fourth point move along $\mathcal{C}'$. Arguing as above, one sees that $\Psi_{[4]}\cdot\varphi_{_2}^*\sigma_1=11$, whence $\beta=1$.
	
	Therefore, $$\varphi_{_2}^*\sigma_1=2H^{[4]}-\frac{1}{2}B^{[4]}\,.$$
	
	Since $\sigma_1$ is ample on $\mathbb{G}$, the lemma follows.
	
\end{proof}

\begin{proof}[Proof of Theorem \ref{theorem:Nef-X4}]
	From Lemma \ref{lemma:NefnessOnS_X}, we see that $$i^*\!\left(2H^{[4]}-\frac{1}{2}B^{[4]}\right)=\left(\psi^{-1}\right)^*\!\left(H^{[2]}\right)$$ is nef on $S_X$. By the projection formula, it follows that $$2H^{[4]}-\frac{1}{2}B^{[4]}$$ non-negatively intersects any curve contained in $S_X$\,. Lemma \ref{lemma:PullbackOfHfromGr} now shows that the above class is nef on all of $X^{[4]}$. By virtue of Theorem \ref{theorem:Nef-X(d-2)}, this class must span an extremal ray of the nef cone of $X^{[4]}$.
\end{proof}

\begin{remark}\label{remark:Nef-X4}
	Theorem \ref{theorem:Nef-X4} provides us with another instance where simple substitution ($d=6$, $n=4$) in the formula (\ref{eq:BHL+Formula}) from Proposition 4.5 of \cite{Bolognese et al.} gives an ample class on $X^{[n]}$, and not an edge of the nef cone.
\end{remark}

\begin{remark}
	We should also remark why the same method does not directly apply for hypersurfaces of degree $d\geq7$. To replicate the method, we would look at the rational map $$\varphi_{_k}:X^{[d-2]}\dashrightarrow \Gr\!\left(h^0\!\left(\mathcal{O}_{\PP^3}(k)\right)-(d-2),H^0\!\left(\mathcal{O}_{\PP^3}(k)\right)\right)$$ taking $Z\in X^{[d-2]}$ to the locus of degree $k$ forms vanishing on $Z$, for some positive integer $k$, and would require that its indeterminacy locus be contained in the locus $S_X$ of collinear schemes. Now, an element $Z$ of $X^{[d-2]}$ would impose $d-2$ independent conditions on the space of all degree $k$ forms on $\PP^3$ only if it does not have a collinear subscheme of length at least $k+2$. Since we want any $Z\in X^{[d-2]}\backslash S_X$ to satisfy this requirement, we need to have $(d-2)-1\leq k+1$, i.e., $d-4\leq k$. Moreover, using arguments similar to those given in Lemmas \ref{lemma:PullbackOfLfromP3*}, \ref{lemma:NefnessOutsideS_X} and \ref{lemma:PullbackOfHfromGr}, one can show (see \cite[Proposition 3.1.(1)]{ABCH}) that in this situation the class $$aH^{[d-2]}-\frac{1}{2}B^{[d-2]}$$ non-negatively intersects any irreducible curve on $X^{[d-2]}$ that is not contained in $S_X$, for any real number $a\geq k$. Thus, in order to show that the class $$\frac{(d^2-d-6)}{2d}H^{[d-2]}-\frac{1}{2}B^{[d-2]}$$ is nef (and hence extremal nef, by Theorem \ref{theorem:Nef-X(d-2)}), we would need $$\frac{d^2-d-6}{2d}\geq k.$$ Putting these two inequalities for $k$ together, we must have $$d-4\leq\frac{d^2-d-6}{2d}\,,$$ whence $d$ must be at most $6$.
	
	An alternative approach could be to stratify $X^{[d-2]}$ into loci where various numbers of points are required to be collinear. But residuation does not work so easily in this method. Moreover, one would need to compute the N\'{e}ron-Severi groups of these loci.
\end{remark}

\section{The nef cone of $X^{[d-r]}$ for a degree $d$ hypersurface}\label{sec:Nef-X-(d-r)}

We begin this section with the observation that if a smooth point $p$ of a hypersurface $X$ is a (ordinary) surface $r$-flex of $X$, then the tangent plane $\mathbb{T}_pX$ to $X$ at $p$ must be the only hyperplane $\Lambda$ in $\PP^3$ that satisfies the condition of Definition \ref{def:SurfaceFlex} above (with respect to $p$).

First, we argue (see Proposition \ref{prop:ExistenceOfSurfaceFlexes}) that for a fixed integer $r\geq3$, once $d$ is large enough, there are $NL$-general degree $d$ hypersurfaces that have ordinary surface $r$-flexes. Then we find a lower bound on the nef cone of $X^{[d-r]}$, assuming that $X$ has an ordinary surface $r$-flex (see Theorem \ref{theorem:Nef-X(d-r)}).

\subsection{Existence of surfaces with $r$-flexes}

In the following lemma, we exhibit a degree $d$ hypersurface having an isolated surface $r$-flex. It will be useful in the proof of Proposition \ref{prop:ExistenceOfSurfaceFlexes} below.

\begin{lemma}\label{lemma:Isolated_r-flex_Example}
	Let $d\geq r\geq3$ be integers.  Let $[x:y:z:w]$ be homogeneous coordinates on $\PP^3$. Then the degree $d$ hypersurface $X_{d,r}$ given as the vanishing locus of $F_{d,r}:=(x^r+y^r)w^{d-r}-zw^{d-1}$ has an ordinary surface $r$-flex at $[0:0:0:1]$ but no other surface $r$-flexes in the affine chart $\{w\neq0\}$.
\end{lemma}

Next, we give an example of a smooth hypersurface with a surface flex, thus illustrating that having a surface flex does not force the surface to have any singularities.

\begin{lemma}\label{lemma:SmoothWith_r-flex_Example}
	Let $d\geq3$, and let $[x:y:z:w]$ be homogeneous coordinates on $\PP^3$. Then the degree $d$ hypersurface $Y_d$ given as the vanishing locus of $G_d:=x^d+y^d+z^d-zw^{d-1}$ is smooth and has a surface $r$-flex at $[0:0:0:1]$, for every $3\leq r\leq d$.
\end{lemma}

The proofs of the two lemmas above are straightforward, and therefore omitted.

\begin{proposition}\label{prop:ExistenceOfSurfaceFlexes}
	Let $r\geq3$ be any integer. Then for any integer $d\geq\binom{r+1}{2}-1$, there is an $NL$-general degree $d$ hypersurface $X$ in $\PP^3$ that has an ordinary surface $r$-flex.
\end{proposition}

\begin{proof}
	
	Choose any $d\geq r$. Recall that there is a projective space $\PP^{N-1}$ parametrizing degree $d$ hypersurfaces in $\PP^3$, where $N=\binom{d+3}{3}$. Let $\mathcal{U}:=\{\,(\Lambda,p)\in\PP^{3*}\times\PP^3\,:\,p\in\Lambda\,\}$ be the universal hyperplane in $\PP^{3*}\times\PP^3$. Consider the incidence correspondence $$\Sigma_{d,r}:=\left\{\,(X,\Lambda,p)\,:\,p\mbox{ is a point of }X\cap \Lambda\mbox{ of multiplicity at least }r\,\right\}\subset\PP^{N-1}\times\mathcal{U}.$$
	
	Let $\mbox{pr}_2:\Sigma_{d,r}\rightarrow\mathcal{U}$ be the second projection. Then $\mbox{pr}_2$ is surjective and each fiber of $\mbox{pr}_2$ is isomorphic to $\PP^{N-1-\binom{r+1}{2}}$. Thus, $\Sigma_{d,r}$ is irreducible of dimension $$N-1-\binom{r+1}{2}+\dim(\mathcal{U})=N+4-\binom{r+1}{2}.$$
	
	Let $\Sigma_{d,r}^{\mbox{\tiny ord}}\subset\Sigma_{d,r}$ be the open locus $$\Sigma_{d,r}^{\mbox{\tiny ord}}:=\left\{\,(X,\Lambda,p)\,:\,p\mbox{ is an ordinary singularity of }X\cap \Lambda\mbox{ of multiplicity exactly }r\,\right\}.$$ We know that it is non-empty, because Lemma \ref{lemma:Isolated_r-flex_Example} shows that $$(X_{d,r}\,,\mathbb{T}_{[0:0:0:1]}X_{d,r}\,,[0:0:0:1])\in\Sigma_{d,r}^{\mbox{\tiny ord}}\,.$$ Hence, $\Sigma_{d,r}^{\mbox{\tiny ord}}$ is irreducible of dimension equal to the dimension of $\Sigma_{d,r}$\,.
	
	Let $\mbox{pr}_1:\Sigma_{d,r}\rightarrow\PP^{N-1}$ be the first projection. Since $X_{d,r}$ has an isolated ordinary surface $r$-flex, it follows that $$\dim\!\left(\mbox{pr}_1\!\left(\Sigma_{d,r}^{\mbox{\tiny ord}}\right)\right)=\dim\!\left(\Sigma_{d,r}^{\mbox{\tiny ord}}\right)=N+4-\binom{r+1}{2}.$$ Of course, $\mbox{pr}_1\!\left(\Sigma_{d,r}^{\mbox{\tiny ord}}\right)$ is also irreducible.
	
	Next, let $D_d$ be the discriminant hypersurface in $\PP^{N-1}$. In other words, the complement $U_d$ of $D_d$ in $\PP^{N-1}$ is the locus of all the smooth degree $d$ hypersurfaces in $\PP^3$. By Lemma \ref{lemma:SmoothWith_r-flex_Example}, it follows that $\mbox{pr}_1\!\left(\Sigma_{d,r}\right)\cap D_d$ is a proper closed subset of the irreducible space $\mbox{pr}_1\!\left(\Sigma_{d,r}\right)$, whence its dimension is strictly less than the dimension of $\mbox{pr}_1\!\left(\Sigma_{d,r}\right)$. Since the dimension of $\mbox{pr}_1\!\left(\Sigma_{d,r}^{\mbox{\tiny ord}}\right)$ is the same as that of $\Sigma_{d,r}$\,, follows that $$\dim\!\left(\mbox{pr}_1\!\left(\Sigma_{d,r}^{\mbox{\tiny ord}}\right)\cap D_d\right)\leq\dim\!\left(\mbox{pr}_1\!\left(\Sigma_{d,r}\right)\cap D_d\right)<\dim\!\left(\mbox{pr}_1\!\left(\Sigma_{d,r}^{\mbox{\tiny ord}}\right)\right).$$
	
	Now, assume $d\geq\binom{r+1}{2}-1$. According to \cite[Theorem 7.5.7]{Carlson et al.}, every component of the Noether-Lefschetz locus $NL_d$ (see Section \ref{sec:Prelims} above) has codimension at least $d-3$ in $U_d$\,. Since the dimension of $U_d$ is $N-1$ and we have assumed $d\geq\binom{r+1}{2}-1$, we find $$\dim(U_d)-(d-3)=N-1-(d-3)<N+4-\binom{r+1}{2}=\dim\!\left(\mbox{pr}_1\!\left(\Sigma_{d,r}^{\mbox{\tiny ord}}\right)\right).$$
	
	Since the ground field $\CC$ is uncountable, $\mbox{pr}_1\!\left(\Sigma_{d,r}^{\mbox{\tiny ord}}\right)$ cannot be covered by countably many subvarieties of strictly smaller dimension. Thus, $$\mbox{pr}_1\!\left(\Sigma_{d,r}^{\mbox{\tiny ord}}\right)\backslash\!\left(D_d\cup NL_d\right)$$ is non-empty. In other words, there is some $NL$-general degree $d$ hypersurface in $\PP^3$ that has an ordinary surface $r$-flex.
\end{proof}

\subsection{Nef divisors when there is an $r$-flex}

By Proposition \ref{prop:ExistenceOfSurfaceFlexes}, the hypothesis of Theorem \ref{theorem:Nef-X(d-r)} is non-vacuous. By exhibiting an orthogonal curve, we compute a lower bound on the nef cone of $X^{[d-r]}$.

\begin{theorem}\label{theorem:Nef-X(d-r)}
	Let $X$ be an $NL$-general degree $d$ hypersurface in $\PP^3$ for some $d\geq5$, and let $r\geq3$. Suppose $X$ has an ordinary surface $r$-flex $p$. Then the nef cone of the Hilbert scheme $X^{[d-r]}$ is spanned by $H^{[d-r]}$ and $$\alpha H^{[d-r]}-\frac{1}{2}B^{[d-r]}$$ for some real number $$\alpha\geq\frac{d^2-d-r^2-r}{2d}\,.$$
\end{theorem}

\begin{proof}
		
		Let $\Omega$ be the curve on $X^{[d-r]}$ given as follows\,: take the hyperplane section $C$ of $X$ cut out by the tangent plane to $X$ at $p$, ``spin'' a line through the multiplicity $r$ point $p$ of the hyperplane section $C$, and take the residual length $d-r$ scheme of intersection of the line with $X$.
		
		This construction gives us a surjective mapping $\tilde{C}\rightarrow\PP^1$ of degree $d-r$, where $\tilde{C}$ is the normalization of $C$. The geometric genus of the projective plane curve $C$ is $$\binom{d-1}{2}-\binom{r}{2}.$$ So by Riemann-Hurwitz, the degree of the ramification divisor of the above mapping is $d^2-d-r^2-r$. It follows that $\Omega\cdot B^{[d-r]}=d^2-d-r^2-r$.
		
		Furthermore, as $H^{[d-r]}$ is represented by the locus of schemes $Z\in X^{[d-r]}$ such that $Z$ meets a fixed general hyperplane section of $X$, it is clear that $\Omega\cdot H^{[d-r]}=d$.
		
		Thus, $$\Omega\!\cdot\!\left(\frac{(d^2-d-r^2-r)}{2d}H^{[d-r]}-\frac{1}{2}B^{[d-r]}\right)=0\,,$$ whence the nef cone of $X^{[d-r]}$ cannot extend past the ray spanned by $$\frac{(d^2-d-r^2-r)}{2d}H^{[d-r]}-\frac{1}{2}B^{[d-r]}.$$
\end{proof}

\section{The nef cone of $X^{[2]}$ for a quintic hypersurface}\label{sec:Nef-X2}

	Let $X$ be an $NL$-general quintic hypersurface in $\PP^3$. By Proposition \ref{prop:ExistenceOfSurfaceFlexes}, there exist $NL$-general quintic hypersurfaces having ordinary surface $3$-flexes, and in this section we compute the nef cone of $X^{[2]}$ for such $X$.

	If, on the other hand, $X$ does not have any surface $3$-flexes, then we find bounds on what the non-trivial \big(i.e., other than the one spanned by $H^{[2]}$\big)  extremal ray of the nef cone of $X^{[2]}$ can be, and observe an interesting phenomenon.

	\begin{theorem}\label{theorem:Nef-X2-withFlex}
		Let $X$ be an $NL$-general quintic hypersurface in $\PP^3$ that has an ordinary surface $3$-flex. Then the nef cone of the Hilbert scheme $X^{[2]}$ is spanned by $H^{[2]}$ and $$\frac{4}{5}H^{[2]}-\frac{1}{2}B^{[2]}\,.$$
	\end{theorem}

	\begin{proof}
	
		From Lemma \ref{lemma:Pullbacks_01}, we know that $$\psi^*i^*\!\left(\frac{1}{5}H^{[3]}\right)=\frac{4}{5}H^{[2]}-\frac{1}{2}B^{[2]}\,,$$ where, as earlier, $\psi:X^{[2]}\xrightarrow{\sim}S_X$ is the residuation isomorphism and $i:S_X\hookrightarrow X^{[3]}$ is the inclusion. It follows that $$\frac{4}{5}H^{[2]}-\frac{1}{2}B^{[2]}$$ is nef on $X^{[2]}$.
	
		By Theorem \ref{theorem:Nef-X(d-r)}, the nef cone of $X^{[2]}$ does not extend past the ray spanned by the above class, whence it must span an extremal ray of the nef cone of $X^{[2]}$.
	
	\end{proof}
	
	On the other hand, when $X$ does not have a surface $3$-flex, we obtain the following result.
	
	\begin{theorem}\label{theorem:Nef-X2-withoutFlex}
		Let $X$ be an $NL$-general quintic hypersurface in $\PP^3$ having no surface $3$-flexes. \begin{enumerate}
			\item The class $$\frac{4}{5}H^{[2]}-\frac{1}{2}B^{[2]}$$ is ample on $X^{[2]}$.\\
			
			\item The nef cone of the Hilbert scheme $X^{[2]}$ is spanned by $H^{[2]}$ and $$\alpha H^{[2]}-\frac{1}{2}B^{[2]}$$ for some real number $\alpha$ with $\dfrac{7}{10}\leq\alpha<\dfrac{4}{5}$.
		\end{enumerate}
	\end{theorem}
	
	\begin{proof}
		
		(i) As usual, let $i:S_X\hookrightarrow X^{[3]}$ be the inclusion, and let $\psi:X^{[2]}\xrightarrow{\sim}S_X$ be the residuation isomorphism.
		
		Since $X$ does not contain any surface $3$-flexes, it follows that $i^*H^{[3]}$ is positive on any curve contained in $S_X$. This can be seen as follows. Let $\Omega$ be any curve on $S_X$. The points of the schemes in $\Omega$ actually sweep out a curve in $X\subset\PP^3$, which must meet a general hyperplane section of $X$. Thus, $\Omega\cdot\left(i^*H^{[3]}\right)>0$.
		
		The composite $S_X\xrightarrow{i} X^{[3]}\xrightarrow{\mbox{\tiny Hilbert-Chow}} X^{(3)}$ is induced by the base-point free divisor $i^*H^{[3]}$, and the above argument shows that it is a finite morphism. It follows that $i^*H^{[3]}$ is ample on $S_X$.
		
		By Lemma \ref{lemma:Pullbacks_01}, we have $$\psi^*i^*\!\left(\frac{1}{5}H^{[3]}\right)=\frac{4}{5}H^{[2]}-\frac{1}{2}B^{[2]}\,,$$ and the first assertion follows.\\
		
		(ii) Consider the curve $\Omega_1$ on $X^{[2]}$ obtained as follows\,: fix a general point $p$ on $X$ and vary a tangent line to $X$ at $p$. For each such tangent line $\ell$, take, along with the fat point $\{p,\ell\}$, one of the ``three other points'' of intersection of $\ell$ with $X$, in turn. All these schemes together give us a curve on $S_X$\,. The curve $\Omega_1$ on $X^{[2]}$ is obtained from this via the residuation isomorphism $\psi$.
		
		Arguing (using Riemann-Hurwitz) as in the proof of Lemma \ref{lemma:i_of_psi_of_Gamma2-IntersectionNumbers}, we get $\Omega_1\cdot B^{[2]}=14$. Also, it is not hard to see that $\Omega_1\cdot H^{[2]}=10$. Thus, we compute $$\Omega_1\!\cdot\!\left(\frac{7}{10}H^{[2]}-\frac{1}{2}B^{[2]}\right)=0,$$ and we obtain the asserted lower bound on the nef cone of $X^{[2]}$.
	\end{proof}
	
	We conclude this section with a remark illustrating a reason as to why Theorem \ref{theorem:Nef-X2-withoutFlex} is interesting.
		
	\begin{remark}\label{remark:Nef-X2}
		If $X$ is an $NL$-general quintic having no surface $3$-flexes, then Theorem \ref{theorem:Nef-X2-withoutFlex} shows that $$\frac{4}{5}H^{[2]}-\frac{1}{2}B^{[2]}$$ is ample and the nef cone extends further (unlike the situation in Theorem \ref{theorem:Nef-X2-withFlex}). This is particularly remarkable, as it shows that $NL$-generality alone is not enough to give a uniform answer for the nef cone of $X^{[2]}$ for quintics $X$, which is in contrast to the cases covered in \cite[Proposition 4.5]{Bolognese et al.}, \cite[Lemma 13.3]{Bayer-Macri-K3} and Section \ref{sec:Nef-X3-X4} above.
	\end{remark}

\section{The nef cone of $X^{[n]}$ for a hypersurface containing a line}\label{sec:Nef-NLspecial}

Now that we have talked about $NL$-general hypersurfaces, we focus on the other end of the spectrum in this section. We explore the computation of the nef cone of $X^{[n]}$ when $X$ is a smooth degree $d\geq3$ hypersurface in $\PP^3$ containing a line (so $X$ is far from being $NL$-general).

\begin{theorem}\label{theorem:Nef-X2-for-NLspecial}
	Let $X$ be a degree $d\geq3$ smooth hypersurface in $\PP^3$ containing a line, and let $n\geq2$. Then the classes $H^{[n]}$ and $$(n-1)H^{[n]}-\frac{1}{2}B^{[n]}$$ span the nef cone of $X^{[n]}$ in the slice of the N\'{e}ron-Severi space spanned by $H^{[n]}$ and $B^{[n]}$.
\end{theorem}

\begin{proof}		
		Let $Z'$ be a closed subscheme of $\PP^3$ of length $\displaystyle\binom{n+2}{3}-n$ consisting of $\displaystyle\binom{n+2}{3}-n$ general points of $\PP^3\backslash X$ so that $$h^0\!\left(\mathcal{O}_{\PP^3}(n-1)\otimes\mathcal{I}_{Z'\subset\PP^3}\right)=n.$$ Let $E$ be the rank $n$ trivial vector bundle on $X^{[n]}$ with fiber $$H^0\!\left(\mathcal{O}_{\PP^3}(n-1)\otimes\mathcal{I}_{Z'\subset\PP^3}\right),$$ and let $F$ be the rank $n$ vector bundle $$\alpha_*\beta^*\!\left(\mathcal{O}_{\PP^3}(n-1)\otimes\mathcal{I}_{Z'\subset\PP^3}\right)$$ on $X^{[n]}$, where $\alpha$ is the composite of the natural maps $\Xi_n\hookrightarrow X\times X^{[n]}\rightarrow X^{[n]}$ and $\beta$ is the composite of the natural maps $\Xi_n\hookrightarrow X\times X^{[n]}\hookrightarrow\PP^3\times X^{[n]}\rightarrow\PP^3$, in which $\Xi_n$ denotes the universal family over $X^{[n]}$.
		
		We have a map $\varphi:E\rightarrow F$ of vector bundles whose action on the fiber over any $Z\in X^{[n]}$ can be described as taking any global section of $\mathcal{O}_{\PP^3}(n-1)\otimes\mathcal{I}_{Z'\subset\PP^3}$ to its restriction to $Z$. The locus over which $\varphi$ fails to be an isomorphism is the divisor $$D:=\left\{\,Z\in X^{[n]}\,:\,h^0\!\left(\mathcal{O}_{\PP^3}(n-1)\otimes\mathcal{I}_{Z'\subset\PP^3}\otimes\mathcal{I}_{Z\subset\PP^3}\right)\neq0\,\right\}$$ on $X^{[n]}$. Note that its class is the first Chern class $c_1(F)$ of $F$. One can use the Grothendieck-Riemann-Roch Theorem as in \cite[Lemma 4.7]{Huizenga12} to find that $$c_1(F)=(n-1)H^{[n]}-\frac{1}{2}B^{[n]}.$$
				
		Now consider the map $$\nu:X^{[n]}\rightarrow G:=\Gr\!\left(\binom{n+2}{3}-n,\binom{n+2}{3}\right)$$ that takes any $Z\in X^{[n]}$ to the locus of degree $n-1$ forms vanishing on $Z$. Let $\sigma_1$ be the Schubert class of all planes in $G$ that meet the subspace of degree $n-1$ forms vanishing on $Z'$. Then $\nu^*\sigma_1$ is the class of $D$, which shows that $$(n-1)H^{[n]}-\frac{1}{2}B^{[n]}$$ is nef on $X^{[n]}$.
		
		Let $\ell$ be a line in $\PP^3$ contained in $X$. The curve in $X^{[n]}$ obtained by fixing $n-1$ points of $\ell$ and letting another point move along $\ell$ is contracted by $\nu$, and it follows that this curve is orthogonal to  $$(n-1)H^{[n]}-\frac{1}{2}B^{[n]}.$$ The assertion follows.

\end{proof}

\bibliographystyle{plain}
  
\end{document}